\documentclass[11pt]{amsart}
\usepackage{mathtools}
\usepackage{tikz}
\usetikzlibrary{cd}
\usepackage{amssymb}
\usepackage{MnSymbol}
\usepackage{float}
\usepackage[normalem]{ulem}
\usepackage{comment}
\usepackage{amsmath}
\usepackage{scalerel}

\usepackage{hyperref}

\tikzset{
  symbol/.style={
    draw=none,
    every to/.append style={
      edge node={node [sloped, allow upside down, auto=false]{$#1$}}}
  }
}

\newtheorem{thm}{Theorem}[section]

\newtheorem{prop}[thm]{Proposition}
\newtheorem{cor}[thm]{Corollary}
\newtheorem*{thm*}{Theorem}
\newtheorem{defn}[thm]{Definition}

\newtheorem*{definition*}{Definition}
\newtheorem{remark}[thm]{Remark}

\numberwithin{equation}{section}

\newcommand{\cont}{{\mathrm{C} }}
\newcommand{\contb}{{\cont}_{\mathrm b}}

\newcommand{\contX}{\cont(X)}
\newcommand{\contbX}{\contb(X)}

\newcommand{\contY}{\cont(Y)}

\newcommand{\contK}{\cont(K)}

\newcommand{\measures}{\mathrm{M}}
\newcommand{\measuresc}{\mathrm{M}_\mathrm{c}}

\newcommand{\vlat}[1]{{#1}}

\newcommand{\borel}[1]{\mathfrak{B}_{#1}}

\newcommand{\di}{\,\mathrm{d}}

\newcommand{\orderdual}[1]{{#1}^\sim}
\newcommand{\ordercontn}[1]{{#1}^\sim_{\mathrm n}}

\newcommand{\ordercontnbidual}[1]{{#1}^{\sim \sim}_{\mathrm{n n}}}
\newcommand{\bidual}[1]{{#1}^{\ast\ast}}
\newcommand{\obidual}[1]{{#1}^{\sim\sim}}

\newcommand{\onefunction}{{\mathbf 1}}

\newcommand{\R}{\mathbb{R}}

\newcommand{\cat}{{\bf C}}

\newcommand{\vlc}{{\bf VL}}
\newcommand{\nvlc}{{\bf NVL}}
\newcommand{\vlic}{{\bf IVL}}

\newcommand{\topc}{{\bf TOP}}

\newcommand{\preann}[1]{\prescript{\circ}{}{#1}}

\newcommand{\defeq}{\ensuremath{\mathop{:}\!\!=}}
\newcommand{\proj}[1]{\underleftarrow{\lim}\hspace{1.5pt} #1 }
\newcommand{\ind}[1]{\underrightarrow{\lim}\hspace{1.5pt} #1 }
\newcommand{\cal}[1]{\mathcal{#1}}

\newcommand{\lbox}{\hspace{1.5pt}\raisebox{0.4pt}{\scaleto{\square}{3.6pt}}\hspace{1.5pt}}

\makeatletter
\newcommand*\bigcdot{\mathpalette\bigcdot@{.5}}
\newcommand*\bigcdot@[2]{\mathbin{\vcenter{\hbox{\scalebox{#2}{$\m@th#1\bullet$}}}}}
\makeatother

\usepackage[UKenglish]{babel}

\usepackage[UKenglish]{datetime}
\usepackage{color}

\begin{document}

\title[The order bidual of $\boldsymbol{\contX}$]{The order bidual of $\boldsymbol{\contX}$ for a realcompact space}

\author{Marcel de Jeu} \address{Mathematical Institute, Leiden University, P.O.\ Box 9512, 2300 RA Leiden, the Netherlands; and Department of Mathematics and Applied Mathematics, University of Pretoria, Cor\-ner of Lynnwood Road and Roper Street,
Hatfield 0083, Pretoria, South Africa }
\email{mdejeu@math.leidenuniv.nl}

\author{Jan Harm van der Walt} \address{Department of Mathematics and Applied Mathematics, University of Pretoria, Cor\-ner of Lynnwood Road and Roper Street,
Hatfield 0083, Pretoria, South Africa}
\email{janharm.vanderwalt@up.ac.za}

\thanks{The second author was supported by the NRF of South Africa, grant number 115047.  The results in this paper were obtained, in part, while the second author visited Leiden University from September 2021 to January 2022.  He thanks the Mathematical Institute of Leiden University for their hospitality.  The visit was funded through the European Union Erasmus+ ICM programme.}


\date{\tt {\today}}




\begin{abstract}
It is well known that the bidual of $\mathrm C(X)$ for a compact space $X$, supplied with the Arens product, is isometrically isomorphic as a Banach algebra to $\mathrm C(\tilde X)$ for some compact space $\tilde X$. The space $\tilde X$ is unique up to homeomorphism. We establish a similar result for realcompact spaces: The order bidual of $\mathrm C(X)$ for a realcompact space $X$, when supplied with the Arens product, is isomorphic as an $f$-algebra to $\mathrm C(\tilde X)$ for some realcompact space $\tilde X$. The space $\tilde X$ is unique up to homeomorphism.
\end{abstract}

\maketitle

\section{Introduction and preliminaries}\label{Section:  Introduction and preliminaries}

Let $K$ be a compact Hausdorff space and denote by $\contK$ the Banach lattice of real-valued continuous functions on $K$.  There exists a (unique up to homeomorphism) compact Hausdorff space $\tilde{K}$ so that $\bidual{\contK}$, the norm bidual of $\contK$, is isometrically lattice isomorphic to $\cont(\tilde{K})$.  This can be seen in various ways.  For instance, invoking two theorems of Kakutani, \cite[Theorems 2 and 15]{Kakutani1941} gives a short proof:  $\contK$ is an AM-space so that $\contK^\ast$ is an AL-space, from which it follows that $\bidual{\contK}$ is a unital AM-space and so is isometrically lattice isomorphic to $\cont(\tilde{K})$ for some compact Hausdorff space $\tilde{K}$.  The uniqueness of $\tilde{K}$ follows from the Banach-Stone Theorem.  Following Dales et al \cite{DalesDashiellLauStrass2016}, we call the space $\tilde{K}$ the \emph{hyper-Stonean envelope} of $K$.

The space $\bidual{\contK}$ is a (real) commutative and unital Banach algebra in\textemdash at first sight\textemdash more than one way.  Pointwise multiplication of functions in $\cont(\tilde{K})$ is one way.  Another is through the Arens products \cite{Arens1951_2,Arens1951_1}; see also \cite[\S 3.1]{DalesDashiellLauStrass2016} for a recent presentation.  We briefly recall the details.

For $\varphi \in \contK^\ast$ and $u\in \contK$,  define $u\cdot \varphi\in \contK^\ast$ by setting
\[
(u \cdot \varphi)(v) \defeq \varphi(uv),~ v\in\contK.
\]
Next, for $\varphi \in \contK^\ast$ and $\Phi \in \bidual{\contK}$, define $\varphi\cdot\Phi \in \contK^\ast$ as
\[
(\varphi\cdot \Phi)(u) \defeq \Phi( u\cdot\varphi),~ u\in\contK.
\]
Finally, for $\Phi,\Psi\in \bidual{\contK}$, we define $\Phi \lbox  \Psi, \Phi \diamond \Psi \in \bidual{\contK}$ as
\[
(\Phi \lbox \Psi)(\varphi) \defeq \Phi(\varphi \cdot \Psi),~ \varphi\in \contK^\ast
\]
and
\[
(\Phi \diamond \Psi)(\varphi) \defeq \Psi(\varphi \cdot \Phi),~ \varphi\in \contK^\ast.
\]

Arens showed in \cite{Arens1951_1} that $\bidual{\contK}$ is a Banach algebra with respect to both $\lbox$ and $\diamond$.  In fact, $\Phi \diamond \Psi = \Phi \lbox \Psi$ for all $\Phi, \Psi \in \bidual{\contK}$, and $(\bidual{\contK},\diamond)$ is a commutative Banach algebra\footnote{An elementary proof of all of this can be found in \cite[Theorem 4.5.5]{DalesDashiellLauStrass2016}. It is true in general that the two Arens products on the bidual of a real $\text{C}^\ast$-algebra, so on the bidual of $\contK$ in particular, coincide; see \cite[Theorem~5.5.3]{Li2003}.}.  Furthermore, $\cont(\tilde{K})$ and $(\bidual{\contK},\diamond)$ are isometrically isomorphic Banach algebras; see for instance \cite[Theorems 5.4.1 and 6.5.1]{DalesDashiellLauStrass2016}.  It follows from \cite[1.6]{GillmanJerison1960} that the isometric algebra isomorphism $\bidual{\contK}\mapsto \cont(\tilde{K})$ is also a lattice isomorphism.  We also note that the canonical isometric embedding $\sigma_K:\contK\to \cont(\tilde{K})$ is both a lattice and an algebra isomorphism onto its range, and $\sigma_K(\onefunction_K)=\onefunction_{\tilde{K}}$.

The following result summarizes the above, and gives some information about the relationship between $K$ and $\tilde{K}$. In it, we denote by $\measures(K)$ the dual of $\contK$ for a compact Hausdorff space $K$, which can be identified with the Radon measures on $K$.  We let $\delta_x$ denote the point measure corresponding to a point $x$ in $K$. We refer the reader to \cite[pp.\ 202--203]{DalesDashiellLauStrass2016} for further details.

\begin{thm*}
Let $K$ be a compact Hausdorff space.
\begin{enumerate}
    \item[(i)] $(\bidual{\contK},\diamond)$ is a unital and commutative Banach lattice algebra.
    \item[(ii)] There exists a unique compact Hausdorff space $\tilde{K}$ so that $\bidual{\contK}$ is isometrically lattice and algebra isomorphic to $\cont(\tilde{K})$.
    \item[(iii)] $\sigma_K:\contK\to \bidual{\contK}$ is an algebra and lattice embedding, and $\sigma_K(\onefunction_K)=\onefunction_{\tilde{K}}$.
    \item[(iv)] There exists a unique continuous surjection $\pi_K:\tilde{K}\to K$ so that $\sigma_K(u) = u\circ \pi_K$ for all $u\in \contK$.
    \item[(v)] $\sigma_K[\contK]=\{u\in\cont(\tilde{K}) \!\!~:~\!\! \text{for every } \tilde{x}\in\tilde{K},~ u \text{ is constant on } \pi_K^{-1}(x)\}$.
    \item[(vi)] Identify $K$ with $\{\delta_x ~:~x\in K\} \subseteq \measures(K)\}$, and $\tilde{K}$ with $\{\delta_{\tilde{x}} ~:~\tilde{x}\in \tilde{K}\}\subseteq \measures(\tilde{K})$.  The topologies on $K$ and $\tilde{K}$, respectively, are the subspace topologies with respect to the weak$^*$-topologies $\sigma(\measures(K),\contK)$ and $\sigma(\measures(\tilde{K}),\cont(\tilde{K}))$, respectively.
    \item[(vii)]  The canonical embedding $\omega:\measures(K)\to \measures(\tilde{K})$ of $\measures(K)$ into its second dual maps $K$ onto the set of isolated points in $\tilde{K}$.
\end{enumerate}
\end{thm*}

Let $X$ be a realcompact space\footnote{We recall that a topological space is \emph{realcompact} if it is homeomorphic to a closed subspace of $\R^{\mathrm{m}}$ for some cardinal $\mathrm{m}$.  Several characterisations of realcompact spaces can be found, for instance, in \cite[Section 3.11]{Engelking1989}.}, and denote by $\contX$ the vector lattice (in fact, $f$-algebra) of real-valued continuous functions on $X$.  Denote by $\orderdual{\contX}$ and $\obidual{\contX}$ the order dual and order bidual or $\contX$, respectively.  The Arens products $\lbox$ and $\diamond$ are defined on $\obidual{\contX}$ in exactly the same way as above.\footnote{This can in fact be done for any lattice ordered algebra, see for instance \cite{HuijsmansdePagter1984II}.}  It is known (see \cite[Remark 3.5 (ii) and Theorem 4.4]{HuijsmansdePagter1984II}) that $(\obidual{\contX},\diamond)$ is a commutative and unital $f$-algebra, the canonical embedding $\sigma_X:\contX\to \obidual{\contX}$ is a lattice and algebra embedding, and $\sigma_X(\onefunction_X)$ is the multiplicative identity in $(\obidual{\contX},\diamond)$

In this paper, we generalize the above theorem to the setting of a realcompact spaces $X$, where the coinciding norm and order bidual of $\contK$ are then replaced with the order bidual of the vector lattice $\contX$ of all continuous real-valued functions on $X$; see Theorem~\ref{Thm:  Main result}. The main point is to show that there exists a realcompact space $\tilde{X}$ (obviously unique up to homeomorphism) so that $(\obidual{\contX},\diamond)$ and $\cont(\tilde{X})$ are lattice and algebra isomorphic.  We solve this problem using the machinery of direct and inverse limits in categories of vector lattices recently developed in \cite{vanAmstelvanderwalt2022}.

The paper is organised as follows.  In the remainder of this section, we give some of the background which is required for the proof our main result. In Section \ref{Subsection:  Realcompact spaces and the order dual of C(X)}, we recall some facts about realcompact spaces and describe the order duals of the associated vector lattices of continuous real-valued functions on them. Pertinent results on lattice homomorphisms, including their relation to ring homomorphisms, are discussed in Section \ref{Subsection:  Lattice homomorphisms}.  Direct and inverse limits are discussed in Section \ref{Subsection:  Direct and inverse limits}.  In particular, we recall the existence results and duality theorems for such limits in suitable categories of vector lattices as developed in \cite{vanAmstelvanderwalt2022}.

Section \ref{Section:  A representation theorem for C(X) order bidual} contains the proof of our main result.  This is done in several steps.  In Section \ref{Subsection:  Direct limits of spaces of measures}, we show how the order dual of $\contX$ for a realcompact space $X$ may be expressed as the direct limit of spaces of measures on compact subsets of $X$; see Corollary~\ref{Cor: Order dual of C(X) is direct limit}.  The relationship between direct limits in the category of topological spaces and inverse limits of the associated spaces of continuous functions is discussed in Section \ref{Subsection:  Inverse limits of C(X)-spaces}. 
In Section \ref{Subsection:  A representation theorem}, the results of Sections \ref{Subsection:  Direct limits of spaces of measures} and \ref{Subsection:  Inverse limits of C(X)-spaces} are combined with the duality result for direct limits recalled in Section \ref{Subsection:  Direct and inverse limits} to yield our main result: a representation theorem for the order bidual of $\cont(X)$ for realcompact $X$; see Theorem~\ref{Thm:  Main result}.

We conclude this introduction with a few notations, conventions, and facts, some of which were already tacitly used above.

Vector lattices are not assumed to be Archimedean. When $\vlat{E}$ is a vector lattice, we denote by $\orderdual{\vlat{E}}$ the order dual of $\vlat{E}$, and by $\ordercontn{\vlat{E}}$ its order continuous dual.  For $u,v\in\vlat{E}$ with $u\leq v$, we let $[u,v]\defeq \{x\in\vlat{E}~:~u\leq x\leq v\}$ denote the corresponding order interval.  An \emph{$f$-algebra} is a vector lattice $E$ which is also an associate algebra so that the product of two positive elements is positive, and for all $u,v\in E$, if $u\wedge v=0$ then $(uw)\wedge v = (wu)\wedge v=0$ for all positive $w\in E$.

All topological spaces in this paper are assumed to be Hausdorff. We recall that a \emph{Tychonoff space} is a completely regular topological space. Every space with the property that every open cover has a countable subcover is realcompact; see \cite[8.2]{GillmanJerison1960}. In particular, every compact space is realcompact. A locally compact space, however, need not be realcompact; see \cite[3.11.2]{Engelking1989}. Every  realcompact space is a Tychonoff space. If $X$ is a topological space, then $\contX$ denotes the space of all continuous, not necessarily bounded, real-valued functions on $X$. A Tychonoff space $X$ is \emph{extremally disconnected} if the closure of every open set in $X$ is open. It is well known that $X$ is extremally disconnected if and only if $\contX$ is Dedekind complete, see for instance \cite[Theorems 43.3 and 43.11]{LuxemburgZaanen1971RSI}. We say that a subspace $Y$ of a topological space $X$ is \emph{$\cont$-embedded in $X$} if every function in $\contY$ can be extended to a function in $\contX$.
If $\theta:X\to Y$ is a continuous map between topological spaces, then we define the lattice homomorphism $T_\theta:\contY\to\contX$ by setting
\[
T_\theta (u)=u\circ \theta,~u\in\contY.
\]
This map is also a unital ring homomorphism.  We note that, in general, we do not require ring homomorphisms to be unital.

\subsection{Realcompact spaces and the order dual of $\boldsymbol{\contX}$}\label{Subsection:  Realcompact spaces and the order dual of C(X)}

Realcompact spaces constitute a convenient class of topological spaces for the study of vector lattices (and rings) of continuous functions.  Indeed, for any topological space $X$ there exists a unique realcompact space, which we denote by $\kappa X$, so that $\contX$ and $\cont(\kappa X)$ are isomorphic as vector lattices and as rings.  This fact plays a important role in the proof of our main result, Theorem \ref{Thm:  Main result}.  As such, we recall it here for the convenience of the reader.

\begin{thm}\label{Thm:  Tychonoff quotient and realcompactification}
Let $X$ be a topological space.  The following statements are true.
\begin{enumerate}
    \item[(i)] There exists a Tychonoff space $Y$ and a continuous surjection $\tau:X\to Y$ so that $\contY\ni u\mapsto u\circ \tau\in \contX$ is a ring and lattice isomorphism.
    \item[(ii)] If $X$ is a Tychonoff space then there exists a realcompact space $\upsilon X$ and an embedding $\upsilon:X\to \upsilon X$ onto a dense, $\cont$-embedded subspace of $\upsilon X$ so that $\cont(\upsilon X)\ni u\mapsto u\circ \upsilon\in \contX$ is a lattice and ring isomorphism. These properties determine $\upsilon X$ up to homeomorphism.
\end{enumerate}
\end{thm}

\begin{thm}\label{Thm:  C(X) determines X}
Let $X$ and $Y$ be realcompact spaces.  If $\contX$ and $\contY$ are ring or lattice isomorphic, then $X$ and $Y$ are homeomorphic.
\end{thm}

\begin{cor}\label{Cor:  Realcompactification of arbitrary space}
Let $X$ be a topological space.  There exists a unique realcompact space $\kappa X$ and a continuous map $\kappa:X\to \kappa X$ onto a dense $\cont$-embedded subspace of $\kappa X$ so that $\cont(\kappa X)\ni u\mapsto u\circ\kappa\in \contX$ is a ring and lattice isomorphism.
\end{cor}

Theorem \ref{Thm:  Tychonoff quotient and realcompactification} is a combination of \cite[3.9 and 8.7]{GillmanJerison1960}, Theorem \ref{Thm:  C(X) determines X} is \cite[8.3]{GillmanJerison1960}, and Corollary \ref{Cor:  Realcompactification of arbitrary space} follows at once from Theorems \ref{Thm:  Tychonoff quotient and realcompactification} and \ref{Thm:  C(X) determines X}. As Corollary \ref{Cor:  Realcompactification of arbitrary space} shows, there is no loss of generality when restricting oneself to realcompact spaces, provided that one is interested solely in the lattice or ring structure of $\contX$ for a topological space $X$.

We now turn to measures on topological spaces, and collect the basic facts which are needed in Section \ref{Section:  A representation theorem for C(X) order bidual}, not the least of which is a description of the order dual of $\contX$ for realcompact $X$ in Theorem~\ref{Thm:  Riesz Representation Theorem for C(X)}.  Since the terminology for measures on topological spaces varies considerably across the literature, we declare our conventions.

Let $X$ be a topological space.  Denote by $\borel{X}$ the Borel $\sigma$-algebra generated by the open sets in $X$.  A \emph{(signed) Borel measure} on $X$ is a real-valued and $\sigma$-additive function on $\borel{X}$.  Following Bogachev \cite[Definition~7.1.1]{Bogachev2007}, we call a Borel measure $\mu$ on $X$ a \emph{Radon measure} if for every $B \in \borel{X}$,
\[
|\mu|(B)=\sup\{|\mu|(C) ~:~ C\subseteq B \text{ is compact}\}.
\]
Equivalently, $\mu$ is Radon if for every $B \in \borel{X}$ and every $\varepsilon>0$ there exists a compact set $C\subseteq B$ so that $|\mu|(B\setminus C)<\varepsilon$. Denote the space of Radon measures on $X$ by $\measures(X)$.  This space is a Dedekind complete vector lattice with respect to the pointwise operations and order; see \cite[Theorem 27.3]{Zaanen1997Introduction} and \cite[Theorem 2.7]{vanAmstelvanderwalt2022}. In fact, for $\mu,\nu\in \measures(X)$, we have
\[
(\mu\vee \nu)(B) = \sup\left\lbrace \mu(A)+\nu(B\setminus A) ~:~ A\subseteq B,~~ A \in \borel{X} \right\rbrace,~~ B \in \borel{X};
\]
and, for any upward directed set $D\subseteq \measures(X)^+$ with supremum $\sup D$ in $\measures(X)$, we have
\[
(\sup D)(B) = \sup\{\mu(B) ~:~ \mu\in D\},~~ B \in \borel{X}.
\]
Recall that the \emph{support} $S_\mu$ of a Borel measure $\mu$ on $X$ is defined as
\[
S_\mu \defeq \{x\in X ~:~ |\mu|(U)>0 \text{ for all } U\ni x \text{ open}\}.
\]
Clearly, $S_\mu$ is a closed subset of $X$.
A non-zero Borel measure $\mu$ may have empty support, and even if $S_\mu\neq \emptyset$, it may have zero measure \cite[Vol.\ II, Example 7.1.3]{Bogachev2007}.  However, if $\mu$ is a nonzero Radon measure then $S_\mu \neq \emptyset$ and $|\mu|(S_\mu)=|\mu|(X)$; in fact, for every $B \in \borel{X}$, $|\mu|(B)=|\mu|(B\cap S_\mu)$, see \cite[Vol.\ II, Proposition 7.2.9]{Bogachev2007}.  We observe that this last identity implies that  $|\mu|(B)=0$ for all $B\in\borel{X}$ disjoint from $S_\mu$.  We list the following useful properties of the support of a measure.  The proofs are straightforward variations of the corresponding results for measures on locally compact spaces, which are discussed in many places, for instance \cite[Chapter 4]{DalesDashiellLauStrass2016}, and are therefore omitted.

\begin{prop}\label{Prop:  Properties of support of a measure}
Let $\mu$ and $\nu$ be Radon measures on a topological space $X$.  The following statements are true.
\begin{enumerate}
    \item[(i)] If $|\mu|\leq |\nu|$ then $S_\mu\subseteq S_\nu$.
    \item[(ii)] $S_{\mu+\nu}\subseteq S_{|\mu|+|\nu|}$
    \item[(iii)] $S_{|\mu|+|\nu|} = S_\mu \cup S_\nu$.
\end{enumerate}
\end{prop}

A Radon measure $\mu$ is called \emph{compactly supported} if $S_\mu$ is compact.  We denote the space of all compactly supported Radon measures on a topological space $X$ as $\measuresc(X)$; this space is an ideal in $\measures(X)$ (see \cite[Theorem 2.7]{vanAmstelvanderwalt2022}), and therefore a Dedekind complete vector lattice in its own right.

The promised description of the order dual of $\contX$ can now be stated.  We refer the reader to \cite[Theorem 2.8]{vanAmstelvanderwalt2022} and the references cited there for more details.

\begin{thm}\label{Thm:  Riesz Representation Theorem for C(X)}
Let $X$ be a realcompact space. For $\mu\in\measuresc(X)$, define $\varphi_\mu\in\orderdual{\contX}$ by setting
\[
\varphi_\mu(u) \defeq \int_X u \di\mu,~~ u\in\contX.
\]
Then the map $\mu\mapsto\varphi_\mu$ is a lattice isomorphism between $\measuresc(X)$ and $\orderdual{\contX}$.
\end{thm}

\subsection{Lattice homomorphisms}\label{Subsection:  Lattice homomorphisms}

Let $\vlat{E}$ and $\vlat{F}$ be vector lattices and $T:\vlat{E}\to\vlat{F}$ a linear map.  $T$ is a \emph{lattice homomorphism} if $T(u\vee v)=T(u)\vee T(v)$ for all $u,v\in \vlat{E}$; $T$ is \emph{normal} if $\sup T[D]=T(\sup D)$ for every $D\subseteq \vlat{E}$ such that $\sup D$ exists in $\vlat{E}$; and $T$ is \emph{interval preserving} if it is positive and $T[[0,u]]=[0,T(u)]$ for every $0\leq u\in \vlat{E}$.  An interval preserving linear map need not be lattice homomorphism, nor is every lattice homomorphism interval preserving; see \cite[p.\ 95]{AliprantisBurkinshaw2006}, for example. However, it is shown in \cite{DingdeJeu2022} that an interval preserving map $T$ is a lattice homomorphism onto an ideal in $\vlat{F}$ whenever its kernel is an ideal in $\vlat{E}$; for injective interval preserving maps this was already observed in \cite[Proposition 2.1]{vanAmstelvanderwalt2022}.
  The following result is a combination of \cite[Theorems 1.73, 2.16, 2.19, and 2.20]{AliprantisBurkinshaw2006}.  It is basic to understanding the duality results for direct limits discussed in Section \ref{Subsection:  Direct and inverse limits}.

\begin{thm}\label{Thm:  Adjoints of interval preserving vs lattice homomorphisms}
Let $\vlat{E}$ and $\vlat{F}$ be Archimedean vector lattices and $T:\vlat{E}\to\vlat{F}$ a positive linear map. Denote by $T^\sim :\vlat{F}^\sim \to \vlat{E}^\sim$ its order adjoint, $\varphi\mapsto \varphi\circ T$.  The following statements are true. \begin{itemize}
    \item[(i)] $T^\sim$ is positive and order continuous.
    \item[(ii)] If $T$ is order continuous then $T^\sim[\ordercontn{\vlat{F}}]\subseteq \ordercontn{\vlat{E}}$.
    \item[(iii)] If $T$ is interval preserving then $T^\sim$ is a lattice homomorphism.
    \item[(iv)] If $T$ is a lattice homomorphism then $T^\sim$ is interval preserving.  The converse is true if $\preann{\orderdual{\vlat{F}}}=\{0\}$.
\end{itemize}
\end{thm}

Let $X$ and $Y$ be topological spaces.  We call a lattice homomorphism $T:\contX\to \contY$ \emph{unital} if $T(\onefunction_X)=\onefunction_Y$.  For example, if  $\theta:X\to Y$ is a continuous map, then
\[
T_\theta:\contY\ni u\mapsto u\circ \theta \in \contX
\]
is a unital lattice homomorphism.  The connection between ring homomorphisms and lattice homomorphisms between spaces of continuous functions, as well as the structure of such maps, is summarized in the following.

\begin{thm}\label{Thm:  Unital lattice homomorphisms}
Let $X$ and $Y$ be Tychonoff spaces.  The following statements are true.
\begin{enumerate}
    \item[(i)] A map $T:\contY\to \contX$ is a unital lattice homomorphism if and only if it is a unital ring homomorphism.
    \item[(ii)] If $Y$ is realcompact and $T:\contY\to \contX$ is a unital lattice homomorphism, then there exists a unique continuous function $\theta:X\to Y$ so that $T(u)=u\circ \theta$ for all $u\in \contY$.
    \item[(iii)] If $\theta:X\to Y$ is a continuous map then $T_\theta :\contY\to \contX$ is surjective if and only if $\theta$ is a homeomorphism onto a $\cont$-embedded subspace of $Y$.
\end{enumerate}
\end{thm}

The statement in (i) is a combination of \cite[1.6]{GillmanJerison1960} and \cite[Corollary 5.5]{HuijsmansdePagter1984}, (ii) is \cite[10.6]{GillmanJerison1960}, and (iii) is \cite[10.3]{GillmanJerison1960}.

\subsection{Direct and inverse limits}\label{Subsection:  Direct and inverse limits}

Let $\cat$ be a category.  Consider an ordered pair $\cal{D}=\left( (O_\alpha)_{\alpha \in I},(e_{\alpha,\beta})_{\alpha \preceq \beta}\right)$ where $I$ is a directed set; $O_\alpha$ is a $\cat$-object for each  $\alpha \in I$; for all $\alpha \preceq \beta$ in $I$, $e_{\alpha,\beta}$ is a $\cat$-morphism from $O_\alpha$ to $O_\beta$; and $e_{\alpha,\alpha}$ is the identity morphism of $O_\alpha$ for all $\alpha\in I$. We call $\cal{D}$ a \emph{direct system} in $\cat$ if the diagram
\[
\begin{tikzcd}[cramped]
O_\alpha \arrow[rd, "e_{\alpha, \beta}"'] \arrow[rr, "e_{\alpha, \gamma}"] & & O_\gamma\\
& O_\beta\arrow[ru, "e_{\beta, \gamma}"']
\end{tikzcd}
\]
commutes in $\cat$ for all $\alpha \preceq \beta\preceq\gamma$ in $I$.  Let $O$ be a $\cat$-object and, for every $\alpha\in I$, let $e_\alpha: O_\alpha \to O$ be a $\cat$-morphism.  The ordered pair $\cal{S} = (O, (e_\alpha)_{\alpha\in I})$ is a  \emph{compatible system of $\cal{D}$} in $\cat$ if the diagram
\[
\begin{tikzcd}[cramped]
O_\alpha \arrow[rd, "e_{\alpha, \beta}"'] \arrow[rr, "e_{\alpha}"] & & O\\
& O_\beta\arrow[ru, "e_{\beta}"']
\end{tikzcd}
\]
commutes in ${\bf C}$ for all $\alpha\preceq\beta$ in $I$.  A \emph{direct limit} of $\cal{D}$ in $\cat$ is a compatible system $\cal{S}=(O, (e_\alpha)_{\alpha\in I})$ of $\cal{D}$ in $\cat$ with the property that, for every compatible system $\cal{S}'=(O', (e_\alpha')_{\alpha\in I})$ of $\cal{D}$ in $\cat$, there exists a unique $\cat$-morphism $r: O\to O'$ so that the diagram
\[
\begin{tikzcd}[cramped]
O \arrow[rr, "r"] & & O'\\
& O_\alpha \arrow[lu, "e_{\alpha}"] \arrow[ru, "e_{\alpha}'"']
\end{tikzcd}
\]
commutes in $\cat$ for every $\alpha\in I$. Direct limits need not exist, but if they do, then they are isomorphic in a strong sense; see, for example, \cite[p.\ 54]{BucurDeleanu1968}. In view of this, if $\cal{D}$ is a direct system in $\cat$ that has a direct limit in $\cat$, then we will simply speak of \emph{the} direct limit of $\cal{D}$, and let $\ind{\cal{D}}$ denote any of its direct limits. If necessary, the context will make clear which one is meant. On occasions, we will just refer to the object in a direct limit as the direct limit.

Now consider an ordered pair $\cal{I} = \left((O_\alpha)_{\alpha \in I},(p_{\beta , \alpha})_{\beta \succeq \alpha}\right)$ where $I$ is a directed set; $O_\alpha$ is a $\cat$-object for each $\alpha\in I$; for all $\beta \succeq \alpha$ in $I$, $p_{\beta , \alpha}$ is a $\cat$-morphism from $O_\beta$ to $O_\alpha$; and $p_{\alpha,\alpha}$ is the identity morphism of $O_\alpha$ for all $\alpha\in I$.  We call $\cal{I}$ an \emph{inverse system} in $\cat$ if the diagram
\[
\begin{tikzcd}[cramped]
O_\gamma \arrow[rd, "p_{\gamma, \beta}"'] \arrow[rr, "p_{\gamma, \alpha}"] & & O_\alpha\\
& O_\beta\arrow[ru, "p_{\beta, \alpha}"']
\end{tikzcd}
\]
commutes in $\cat$ for all $\alpha \succeq\beta\succeq\gamma$ in $I$.  If $O$ is a $\cat$-object and $p_\alpha:O\to O_\alpha$ is a $\cat$-morphism for every $\alpha \in I$, then the ordered pair $\cal{S} =  \left(O,(p_\alpha)_{\alpha\in I}\right)$ is a \emph{compatible system of $\cal{I}$} in $\cat$ if the diagram
\[
\begin{tikzcd}[cramped]
O \arrow[rd, "p_\beta"'] \arrow[rr, "p_{\alpha}"] & & O_\alpha\\
& O_\beta\arrow[ru, "p_{\beta, \alpha}"']
\end{tikzcd}
\]
commutes in $\cat$ for all $\alpha \preceq \beta$ in $I$.  A compatible system $\cal{S}=\left(O,(p_\alpha)_{\alpha\in I}\right)$ of $\cal{I}$ in $\cat$ is an \emph{inverse limit of $\cal{I}$} in $\cat$ if, for every compatible system $\cal{S}'=(O',(p_\alpha')_{\alpha \in I})$ of $\cal{I}$ in $\cat$, there exists a unique $\cat$-morphism $s:O'\to O$ so that the diagram
\[
\begin{tikzcd}[cramped]
O' \arrow[rd, "p_{\alpha}'"']\arrow[rr, "s"] & & O\arrow[ld, "p_{\alpha}"]\\
& O_\alpha
\end{tikzcd}
\]
commutes in $\cat$ for every $\alpha \in I$.  As with direct limits, inverse limits need not exist, but they are isomorphic in a strong sense if they do. Also similarly, if $\cal {I}$ has an inverse limit in $\cat$, then we will simply speak of \emph{the} inverse limit, and let $\proj{\cal{I}}$ denote any of its inverse limits, or the object in any of these.

We will be concerned with limits of direct and inverse systems in the following categories.

\begin{table}[H]

\begin{tabular}{ |l|l|l| }
\hline
\textsc{Category} \quad & \textsc{Objects}\quad & \textsc{Morphisms} \quad \\
\hline
$\vlc$ & Vector lattices & Lattice homomorphisms\\
\hline
$\nvlc$ & Vector lattices & Normal lattice homomorphisms\\
\hline
$\vlic$ & Vector lattices & Interval preserving lattice homomorphisms\\
\hline
$\topc$ & Topological spaces & Continuous functions \\
\hline
\end{tabular}
\smallskip
\caption{}
\end{table}

We quote the following existence result for direct and inverse limits in the categories listed in Table 1.  For the categories of vector lattices these can be found in \cite{Filter1988,vanAmstelvanderwalt2022}; inverse limits in $\topc$ are discussed in \cite[Appendix 2]{Dugundji1966}.

\begin{thm}\label{Thm:  Existence of limits}
The following statements are true.
\begin{enumerate}
    \item[(i)] If $\cat$ is any of the categories $\vlc$, $\vlic$ or $\topc$, and $\cal{D}$ is a direct system in $\cat$, then $\ind{\cal{D}}$ exists in $\cat$.
    \item[(ii)] If $\cal{I}=\left((\vlat{E}_\alpha)_{\alpha \in I},(p_{\beta , \alpha})_{\beta \succeq \alpha}\right)$ is an inverse system in $\vlc$, then $\proj{\cal{I}}$ exists in $\vlc$. Take such an inverse limit $\left(\vlat{E},(p_\alpha)_{\alpha\in I}\right)$. Then the lattice homomorphisms $p_\alpha:\vlat{E}\to \vlat{E}_\alpha$ separate the points of $O$.
    \item[(iii)] Let $\cal{I}=\left((\vlat{E}_\alpha)_{\alpha \in I},(p_{\beta , \alpha})_{\beta \succeq \alpha}\right)$ be an inverse system in $\nvlc$, and $\cal{S}=\left(\vlat{E},(p_\alpha)_{\alpha \in I}\right)$ an inverse limit in $\vlc$.  If $\vlat{E}_\alpha$ is Dedekind complete for all $\alpha \in I$, then $\cal{S}$ is an inverse limit of $\cal{I}$ in $\nvlc$.
\end{enumerate}
\end{thm}

\begin{remark}\label{Remark:  Structure of direct limits}
The following facts from \cite[Appendix 2]{Dugundji1966} and \cite[Remark 3.7]{vanAmstelvanderwalt2022} will be used on a number of occasions.  Let $\cat$ be any of the categories in Table 1 and $\cal{D}=\left((O_\alpha)_{\alpha \in I},(e_{\alpha,\beta})_{\alpha \preceq \beta}\right)$ a direct system in $\cat$.  Suppose that $\left(O,(e_\alpha)_{\alpha\in I}\right)$ is a direct limit of $\cal{D}$ in $\cat$.  Then the following is true.
\begin{enumerate}
    \item[(i)] For every $x\in O$, there exists an $\alpha \in I$ and $x_\alpha \in O_\alpha$ so that $x=e_\alpha(x_\alpha)$.  If also $x = e_\beta (x_\beta)$ for some $\beta\in I$ and $x_\beta\in O_\beta$, then there exists $\gamma \succeq \alpha,\beta$ in $I$ so that $e_{\alpha , \gamma}(x_\alpha) = e_{\beta , \gamma}(x_\beta)$, and hence
        \[
        e_\gamma ( e_{\alpha , \gamma}(x_\alpha)) = x = e_\gamma ( e_{\beta , \gamma}(x_\beta)).
        \]
    \item[(ii)] If $P$ is a $\cat$-object and $f:O\to P$ is a map, then $f$ is a $\cat$-morphism if and only if $f\circ e_\alpha$ is a $\cat$-morphism for every $\alpha \in I$.
\end{enumerate}
\end{remark}

We end this section with a basic duality results for direct limits, given in \cite[Theorem 5.4]{vanAmstelvanderwalt2022}


\begin{thm}\label{Thm:  Dual of ind sys in VLIC is proj of duals in NVL}
Let $\cal{D}\defeq \left( (\vlat{E}_\alpha)_{\alpha\in I}, (e_{\alpha, \beta})_{\alpha\preceq\beta}\right)$ be a direct system in $\vlic$, and let $\cal{S} \defeq \left( \vlat{E},(e_\alpha)_{\alpha\in I}\right)$ be a direct limit of $\cal{D}$ in $\vlic$.  Define
\[
\cal{D}^\sim \defeq \left( (\vlat{E}^\sim_\alpha)_{\alpha\in I}, (e^\sim_{\alpha, \beta})_{\alpha\preceq\beta}\right)
\]
and
\[
\cal{S}^\sim \defeq \left( \vlat{E}^\sim, (e^\sim_{\alpha})_{\alpha\preceq\beta}\right).
\]
Then $\cal{D}^\sim$ is an inverse system in $\nvlc$, and $\proj{\cal{D}^\sim}= \cal{S}^\sim$ in $\nvlc$.
\end{thm}

\section{A representation theorem for $\contX^{\sim\sim}$ for realcompact $X$}\label{Section:  A representation theorem for C(X) order bidual}

We will now start with the development of preparatory material in Sections~\ref{Subsection:  Direct limits of spaces of measures} and~\ref{Subsection:  Inverse limits of C(X)-spaces}, before establishing our main result in Section~\ref{Subsection:  A representation theorem}.

\subsection{Direct limits of spaces of measures}\label{Subsection:  Direct limits of spaces of measures}

Our first preparatory topic consists of direct limits of spaces of measures. We show, amongst others, that the order dual of $\contX$ for realcompact $X$ is the direct limit of spaces of measures on compact subsets of $X$; see Corollary~\ref{Cor: Order dual of C(X) is direct limit}.

Let $X$ be a topological space and $K\subseteq X$ a compact set.  For $\mu\in\measures(K)$, define
\[
\mu^{(K,X)}(B) \defeq \mu (B\cap K),~~ B\in\borel{X}.\]
For $\nu\in \measures(X)$, define
\[
\nu_{|K}(B) \defeq \nu (B),~~ B\in\borel{K}.
\]

\begin{prop}\label{Prop:  Extension and retraction of measures}
Let $X$ be a topological space and $K\subseteq X$ a compact set. The following statements are true.
\begin{enumerate}
    \item[(i)] If $\mu\in \measures(K)$, then $\mu^{(K,X)}\in \measuresc(X)$ and $S_{\mu^{(K,X)}}\subseteq K$.
    \item[(ii)] If $\nu\in\measures(X)$, then $\nu_{|K}\in\measures(K)$.
    \item[(iii)] The map $T:\measures(K)\ni \mu\mapsto \mu^{(K,X)}\in \measuresc(X)$ is injective and interval preserving, and a lattice isomorphism onto the ideal $T[\measures(K)]$ in $\measuresc(X)$.
    \item[(iv)] $T[\measures(K)]$ is, in fact, a band in $\measuresc(X)$. More precisely,
    \[
    T[\measures(K)]=\{\nu \in \measuresc(X) ~:~ S_\nu \subseteq K\}.
	\]
    \item[(v)] $T^\sim : \orderdual{\measuresc(X)}\to \orderdual{\measures(K)}$ is surjective.
\end{enumerate}
\end{prop}

\begin{proof}
For the proof of (i), fix $\mu\in\measures(K)$.  It is clear that $\mu^{(K,X)}$ is a Borel measure on $X$.  We show that $\mu^{(K,X)}$ is a Radon measure.  For the moment, suppose that $\mu\geq 0$.  Fix some Borel set $B$ in $X$ and a real number $\varepsilon>0$.  Then $B\cap K\in\borel{K}$; see \cite[Vol.\ II, Lemma 6.2.4]{Bogachev2007}.  There exists a compact set $C\subseteq B\cap K\subseteq B$ so that $\mu((B\cap K)\setminus C)<\varepsilon$.  But $(B\setminus C)\cap K = (B\cap K)\setminus C$ so that $\mu^{(K,X)}(B\setminus C) = \mu((B\setminus C)\cap K) = \mu((B\cap K)\setminus C)<\varepsilon$.  Therefore $\mu^{(K,X)}\in \measures(X)$.

Let $\mu$ be an arbitrary Radon measure on $K$.  Then $\mu = \mu^+ - \mu^-$ with both $\mu^+$ and $\mu^-$ positive Radon measures on $K$.  Consequently, by the above argument, $\mu^{+ (K,X)}$ and $\mu^{- (K,X)}$ are Radon measures on $X$.  But $\mu^{(K,X)}= \mu^{+ (K,X)} - \mu^{- (K,X)}$, as one may readily verify, so that $\mu^{(K,X)}\in \measures(X)$.

We claim that $S_{\mu^{(K,X)}}\subseteq K$. Let $x\in X\setminus K$.  Since $X$ is a Hausdorff space, there exists an open neighbourhood $V$ of $x$ which is disjoint from $K$.  For such a $V$,
\[
\left(\mu^{(K,X)}\right)^+(V) = \sup\{\mu(A\cap K) ~:~ A\in\borel{X},~ A\subseteq V\}=0
\]
and
\[
\left(\mu^{(K,X)}\right)^-(V) = \sup\{-\mu(A\cap K) ~:~ A\in\borel{X},~ A\subseteq V\}=0
\]
so that $|\mu^{(K,X)}|(V)=0$.  Hence $x\in X\setminus S_{\mu^{(K,X)}}$, showing that $S_{\mu^{(K,X)}}\subseteq K$, as claimed. It is now clear that $S_{\mu^{(K,X)}}$ is compact.

For (ii), let $\nu\in\measures(X)$.  We note that $\borel{K}=\{B\cap K ~:~ B\in\borel{X}\}\subseteq \borel{X}$; see \cite[Vol.\ II, Lemma 6.2.4]{Bogachev2007}.  Hence $\nu_{|K}$ is the restriction of $\nu$ to $\borel{K}$.  It is now immediate that $\nu_{|K}$ is a Radon measure on $K$.

We proceed with the proof of (iii).  It is clear that $T$ is linear and positive.  To see that $T$ is injective, let $T(\mu) = \mu^{(K,X)}=0$ for some $\mu\in\measures(K)$.  For $B\in\borel{K}\subseteq \borel{X}$ we have $\mu(B) = \mu(B\cap K)=T(\mu)(B)=0$ so that $\mu=0$.

To see that $T$ is interval preserving, let $0\leq \mu\in\measures(K)$ and $\nu\in\measuresc(X)$ satisfy $0\leq \nu\leq T(\mu)$.  We show that $0\leq \nu_{|K}\leq \mu$ and $T(\nu_{|K})=\nu$.  For every $B\in\borel{K}\subseteq \borel{X}$, $0\leq \nu_{|K}(B)=\nu(B) \leq T(\mu)(B) = \mu(B\cap K)=\mu(B)$ so that $0\leq \nu_{|K}\leq \mu$.  If $A\in\borel{X}$, then $T(\nu_{|K})(A) = \nu_{|K}(A\cap K) = \nu(A\cap K)$.  Since $0\leq \nu\leq T(\mu)$ and $S_{T(\mu)}\subseteq K$ by (i), it follows from Proposition \ref{Prop:  Properties of support of a measure} (i) that $S_{\nu}\subseteq K$.  Therefore $\nu(A)=\nu(A\cap K)=T(\nu_{|K})(A)$, showing that $T(\nu_{|K})=\nu$, as claimed.

As already mentioned in the introduction, the fact that $T$ is injective and interval preserving implies that it is a lattice isomorphism onto an ideal in $\measuresc(X)$.

We now prove (iv).  It follows from (i) that $T[\measures(K)]\subseteq \{\nu \in \measuresc(X) ~:~ S_\nu \subseteq K\}$.  For the reverse inclusion, first take $0\leq \nu \in \{\nu \in \measuresc(X) ~:~ S_\nu \subseteq K\}$.  Then $\nu_{|K}\in \measures(K)$ by (ii) and, for $B\in\borel{X}$, $T(\nu_{|K})(B) = \nu_{|K}(B\cap K) = \nu(B\cap K) = \nu(B)$; the final identity follows from the fact that $S_\nu \subseteq K$.  Hence $\nu\in T[\measures(K)]$. For general $\nu \in \{\nu \in \measuresc(X) ~:~ S_\nu \subseteq K\}$, this then follows upon decomposing $\nu$ into its positive and negative parts.  Therefore $T[\measures(K)]\subseteq \{\nu \in \measuresc(X) ~:~ S_\nu \subseteq K\}$.

   It remains to show that it is a band, i.e., that $\{\nu \in \measuresc(X) ~:~ S_\nu \subseteq K\}$ is a band in $\measuresc(X)$.  Suppose that $D\subseteq  \{\nu \in \measuresc(X) ~:~ S_\nu \subseteq K\}^+$ and $D \uparrow \mu$ in $\measuresc(X)$.  Then $\mu(X\setminus K) = \sup\{\nu(X\setminus K) ~:~ \nu\in D\}=0$ so that $S_\mu \subseteq K$; that is, $\mu\in\{\nu \in \measuresc(X) ~:~ S_\nu \subseteq K\}$.

That (v) is true follows from (iii) and (iv).  Indeed, let $P$ be the band projection of $\measuresc(X)$ onto $T[\measures(K)]$. For $\varphi \in \orderdual{\measuresc(K)}$, define $\psi \defeq  \varphi \circ T^{-1}\circ P$.  Then  $\psi\in\orderdual{\measuresc(X)}$ and $T^\sim(\psi) = \varphi$.
\end{proof}

The following result will be used repeatedly in the proof of our main result, Theorem \ref{Thm:  Main result}.

\begin{prop}\label{Prop:  Measure extension operator adjoint}
Let $X$ be a realcompact space, and let $K$ be a compact subset of $X$. Consider the map $T:\measures(K)\ni \mu\mapsto \mu^{(K,X)}\in\measuresc(X)$. Identify $\obidual{\contX}$ with $\orderdual{\measuresc(X)}$ and $\obidual{\contK}$ with $\orderdual{\measures(K)}$.  Denote by $\sigma_X:\contX\to \orderdual{\measuresc(X)}$ and $\sigma_K:\contK\to \orderdual{\measures(K)}$ the canonical embeddings of $\contX$ and $\contK$ into their respective order biduals.  Then $T^\sim(\sigma_X(\onefunction_X))=\sigma_K(\onefunction_K)$.
\end{prop}

\begin{proof}
Let $\mu \in \measures(K)$.  Then
\[
T^\sim(\sigma_X(\onefunction_{X}))(\mu) = \displaystyle \int_{X} \onefunction_{X}\di T(\mu)= T(\mu)(X).
\]
But by definition of $T$, $T(\mu)(X)=\mu(K)$.  Therefore
\[
T^\sim(\sigma_X(\onefunction_{X}))(\mu)= \mu(K) = \int_{K}\onefunction_{K}\di\mu = \sigma_{K}(\onefunction_{K})(\mu).
\]
Hence $T^\sim(\sigma_X(\onefunction_X))=\sigma_K(\onefunction_K)$.
\end{proof}

We  now define the principal object of this subsection.

\begin{defn}\label{Def:  Inverse system of measures}
Let $X$ be a topological space, and let $\cal{K}=\{K_\alpha ~:~ \alpha\in I\}$ be the collection of all nonempty and compact subsets of $X$, indexed by a directed set $I$ so that $K_\alpha \subseteq K_\beta$ if and only if $\alpha \preceq \beta$.  For $\alpha \preceq \beta$ in $I$, define $T_{\alpha , \beta}: \measures(K_\alpha)\to \measures(K_\beta)$ by setting $T_{\alpha,\beta}(\mu) \defeq \mu^{(K_\alpha,K_\beta)}$ for $\mu\in\measures(K_\alpha)$.  Likewise, for $\alpha\in I$, define $T_\alpha:\measures(K_\alpha)\to \measuresc(X)$ by setting $T_\alpha(\mu)\defeq  \mu^{(K_\alpha,X)}$ for $\mu\in\measures(K_\alpha)$.
Finally, define
\[
\cal{D}_{X}\defeq \left(\left(\measures(K_\alpha)\right)_{\alpha\in I},\left(T_{\alpha,\beta}\right)_{\alpha \preceq \beta}\right)
\]
and
\[
\cal{S}_{X}\defeq \left(\measuresc(X),\left(T_\alpha\right)_{\alpha\in I}\right).
\]
\end{defn}

Naturally, one can take $\cal{K}$ for $I$. Introducing $I$ and indices $\alpha$, however, yields a notation that fits better with that in other results later on.

\begin{prop}\label{Prop:  Direct systems of Radon measures}
Let $X$ be a topological space.  The following statements are true.
\begin{enumerate}
	\item[(i)] $\cal{D}_{X}$ is a direct system in $\vlic$.
	\item[(ii)] $\cal{S}_{X} = \ind{\cal{D}_{X}}$ in $\vlic$.
\end{enumerate}
\end{prop}

\begin{proof}
It follows immediately from Proposition \ref{Prop:  Extension and retraction of measures} (iii) and the definitions of the $T_{\alpha,\beta}$ and $T_\alpha$, respectively, that $\cal{D}_{X}$ is a direct system in $\vlic$ and that $\cal{S}_{X}$ is a compatible system of $\cal{D}_{X}$ in $\vlic$.  We show that $\cal{S}_{X} = \ind{\cal{D}_{X}}$ in $\vlic$.

It follows from Theorem \ref{Thm:  Existence of limits} (i) that $\ind{\cal{D}_{X}}$ exists in $\vlic$.  Take such a direct limit $\left(\vlat{F},(S_\alpha)_{\alpha\in I}\right)$.  Because $\cal{S}_{X}$ is a compatible system of $\cal{D}_{X}$ in $\vlic$, there exists a unique interval preserving lattice homomorphism $T:\vlat{F}\to \measuresc(X)$ so that $T_\alpha = T\circ S_\alpha$ for every $\alpha \in I$.  We claim that $T$ is a lattice isomorphism.  From this it follows that $\cal{S}_{X} = \ind{\cal{D}_{X}}$ in $\vlic$.

It suffices to show that $T$ is bijective.  Suppose that $T(u)=0$ for some $u\in\vlat{F}$.  According to Remark \ref{Remark:  Structure of direct limits} (i), there exists an $\alpha \in I$ and $\mu_\alpha \in \measures(K_\alpha)$ so that $u=S_\alpha(\mu_\alpha)$.  Then $T_\alpha(\mu_\alpha) = T(u)=0$, so that $\mu_\alpha=0$ by Proposition \ref{Prop:  Extension and retraction of measures} (iii).  Therefore $u=S_\alpha(\mu_\alpha)=0$, so that $T$ is injective.

Take $\mu\in\measuresc(X)$, and let $K_\alpha$ be a compact subset of $X$ so that $S_{\mu}\subseteq K_\alpha$.  It follows from Proposition \ref{Prop:  Extension and retraction of measures} (ii) that $\mu_{|K_\alpha}\in \measures(K_\alpha)$.  Let $B\in \borel{X}$.  Then, since $S_\mu\subseteq K_\alpha$, $T_\alpha(\mu_{|K_\alpha})(B) = \mu_{|K_\alpha}(B\cap K_\alpha) = \mu(B\cap K_\alpha) = \mu(B)$.  Therefore $T_\alpha (\mu_{|K_\alpha}) = \mu$.  From this it follows that $\mu=T_\alpha(\mu_{|K_\alpha}) = T(S_\alpha(\mu_{|K_\alpha}))$. Hence $T$ is surjective, which completes the proof.
\end{proof}

In view of Theorem~\ref{Thm:  Riesz Representation Theorem for C(X)}, the previous result has the following consequence.

\begin{cor}\label{Cor: Order dual of C(X) is direct limit}
Let $X$ be a realcompact space.  Then $\orderdual{\contX}$ is the direct limit of $\cal{D}_{X}$ in $\vlic$.
\end{cor}


We collect one more consequence of Proposition~\ref{Prop:  Direct systems of Radon measures}, using the notation of Theorem~\ref{Thm:  Dual of ind sys in VLIC is proj of duals in NVL}.

\begin{cor}\label{Cor:  Order dual of compactly supported Radon measures}
Let $X$ be a topological space.  Then $\cal{D}_{X}^\sim$ is an inverse system in $\nvlc$, and $(\orderdual{\measuresc(X)},(T^\sim_\alpha)_{\alpha \in I})=\proj{\cal{D}_{X}^\sim}$ in $\nvlc$.  Every functional $\varphi \in \orderdual{\measuresc(X)}$ is order continuous.
\end{cor}

\begin{proof}
It follows immediately from Proposition \ref{Prop:  Direct systems of Radon measures} and Theorem \ref{Thm:  Dual of ind sys in VLIC is proj of duals in NVL} that $\cal{D}_{X}^\sim$ is an inverse system in $\nvlc$, and that $(\orderdual{\measuresc(X)},(T^\sim_\alpha)_{\alpha \in I})=\proj{\cal{D}_{X}^\sim}$.  It remains to show that every functional $\varphi \in \orderdual{\measuresc(X)}$ is order continuous.

Let $D\downarrow 0$ in $\measuresc(X)$.  Without loss of generality, we may suppose that $D$ is bounded from above.  Since the $T_\alpha$ are interval preserving, it follows from Remark \ref{Remark:  Structure of direct limits} (i) that there exists a $\beta\in I$ so that $D\subseteq T_\beta[\measures(K_\beta)]$.  Let $D_\beta \defeq T_\beta^{-1}[D]$.  According to Proposition~\ref{Prop:  Extension and retraction of measures} (iii), $T_\beta$ is an order isomorphism onto an ideal in $\measuresc(X)$.  Hence $D_\beta \downarrow 0$ in $\measures(K_\beta)$.  Let $0\leq \varphi \in \orderdual{\measuresc(X)}$.  Then $\varphi[D] = \varphi[T_\beta[D_\beta]] = T_\beta^\sim (\varphi)[D_\beta]$.  It follows from Theorem~\ref{Thm:  Riesz Representation Theorem for C(X)} and Kakutani's duality results that $\measures(K_\beta)$ is an AL-space with respect to the total variation norm, so that every order bounded functional on $\measures(K_\beta)$ is order continuous.  Therefore $\inf \varphi[D] = \inf T_\beta^\sim (\varphi)[D_\beta]=0$, which shows that $\varphi$ is order continuous.
\end{proof}

\subsection{Inverse limits of $\boldsymbol{\contX}$-spaces}\label{Subsection:  Inverse limits of C(X)-spaces}

We now turn to inverse limits of $\contX$-spaces. The study of inverse systems of such spaces in categories of vector lattices is intimately connected to that of direct systems in $\topc$ via the natural contravariant functor. This functor assigns $\contX$ to a topological space $X$; to a continuous map $\theta:X\to Y$ between topological spaces it assigns the lattice homomorphism $T_\theta:\cont(Y)\to\cont(X)$ that is defined by setting $T_\theta(u)\defeq u\circ\theta$ for $u\in\cont(Y)$.

As the following result shows, this functor maps direct limits in $\topc$ to inverse limits in $\vlc$. Although the issue seems fairly natural to consider, we are not aware of a reference for this fact.

\begin{prop}\label{Prop:  Continuous functions on direct limit is inverse limit}
Let $\cal{D}=\left( (X_\alpha)_{\alpha \in I},(\theta_{\alpha,\beta})_{\alpha\preceq \beta}\right)$ be a direct system in $\topc$ with direct limit $\cal{I} = (X,(\theta_\alpha)_{\alpha\in I})$. For all $\beta \succeq \alpha$ in $I$, define $T_{\beta , \alpha}:\cont(X_\beta)\to\cont(X_\alpha)$ by setting $T_{\beta,\alpha}\defeq T_{\theta_{\alpha,\beta}}$. Likewise, for $\alpha\in I$, define  $T_\alpha:\contX\to\cont(X_\alpha)$ by setting $T_\alpha\defeq T_{\theta_\alpha}$. Finally, define
\[
\cal{D}^\star \defeq \left( (\cont(X_\alpha))_{\alpha \in I},(T_{\beta , \alpha})_{\beta\succeq \alpha}\right)
\]
and
\[
\cal{I}^\star \defeq \left(\contX , (T_\alpha)_{\alpha \in I}\right).
\]
\begin{enumerate}
	\item[(i)] $\cal{D}^\star$ is an inverse system in $\vlc$.
	\item[(ii)] $\cal{I}^\star=\proj \cal{D}^\star$ in $\vlc$.
\end{enumerate}
\end{prop}

\begin{proof}
Let $\gamma \succeq \beta\succeq \alpha$ in $I$.  For every $u\in \cont(X_\gamma)$,
\[
T_{\gamma,\alpha}(u) = u\circ \theta_{\alpha,\gamma} = u\circ (\theta_{\beta,\gamma}\circ \theta_{\alpha,\beta}) = (u\circ \theta_{\beta,\gamma})\circ \theta_{\alpha,\beta} = T_{\beta,\alpha}(T_{\gamma,\beta}(u)).
\]
Similarly, $T_\alpha = T_{\beta,\alpha}\circ T_\beta$ for all $\beta \succeq \alpha$ in $I$.  Therefore $\cal{D}^\star$ is an inverse system in $\vlc$ and $\cal{I}^\star$ is a compatible system of $\cal{D}^\star$ in $\vlc$.

Take an inverse limit $\left(\vlat{F},(S_\alpha)_{\alpha\in I}\right)$ of $\cal{D}^\star$ in $\vlc$.  Then there exists a unique lattice homomorphism $T:\contX \to \vlat{F}$ so that $T_\alpha = S_\alpha \circ T$ for all $\alpha \in I$.  We show that $T$ is a bijection, hence a lattice isomorphism.  From this it follows that $\proj{\cal{D}^\star} = \cal{I}^\star$ in $\vlc$.

First we show that $T$ is injective.  Let $u\in \contX$.  If $T(u)=0$ then $T_\alpha(u) = S_\alpha(T(u))=0$ for every $\alpha\in I$.  For every $x\in X$, there exist $\alpha \in I$ and $x_\alpha \in X_\alpha$ so that $x=\theta_\alpha(x_\alpha)$; see Remark \ref{Remark:  Structure of direct limits} (i).  Hence $u(x) = u(\theta_\alpha(x_\alpha)) = T_\alpha(u)(x_\alpha)=0$.  Therefore, if $T(u)=0$ then $u=0$; thus $T$ is injective.

To see that $T$ is surjective, consider $w\in \vlat{F}$.  We define a function $u:X\to \R$ as follows:  For every $x\in X$, select $\alpha\in I$ and $x_\alpha \in X_\alpha$ so that $x=\theta_\alpha(x_\alpha)$.  We set $u(x) \defeq  S_\alpha(w)(x_\alpha)$.

We verify that, for a given $x\in X$, $u(x)$ is independent of the choice of $\alpha \in I$ and $x_\alpha \in X_\alpha$.  Suppose that $\beta \in I$ and $x_\beta \in X_\beta$ satisfy $\theta_\beta (x_\beta) = x$.  According to Remark \ref{Remark:  Structure of direct limits} (i), there exists $\gamma \succeq \alpha , \beta$ in $I$ and $x_\gamma \in X_\gamma$ so that $\theta_{\alpha , \gamma}(x_\alpha) = x_\gamma = \theta_{\beta , \gamma} (x_\beta)$.  Since $S_\alpha = T_{\gamma , \alpha}\circ S_\gamma$ and $S_\beta = T_{\gamma , \beta}\circ S_\gamma$ it follows that
\[
S_\alpha(w)(x_\alpha) = T_{\gamma , \alpha}(S_\gamma(w))(x_\alpha) = S_\gamma (w)(\theta_{\alpha , \gamma}(x_\alpha) = S_\gamma(w)(x_\gamma)
\]
and
\[
S_\beta(w)(x_\beta) = T_{\gamma , \beta}(S_\gamma(w))(x_\beta) = S_\gamma (w)(\theta_{\beta , \gamma}(x_\beta)) = S_\gamma(w)(x_\gamma).
\]
Therefore $S_\alpha(w)(x_\alpha) = S_\beta(w)(x_\beta)$, as required.

Next we show that $u$ is continuous.  Fix $\alpha \in I$.  The definition of $u$ implies that, for every $x\in X_\alpha$, $u(\theta_\alpha(x)) = S_\alpha(w)(x)$.  That is, $u\circ \theta_\alpha = S_\alpha(w)$ which is continuous.  Since this holds for every $\alpha \in I$, $u$ is continuous by Remark \ref{Remark:  Structure of direct limits} (ii).

It remains to verify that $T(u)= w$.  By Theorem~\ref{Thm:  Existence of limits} (ii), this is equivalent to the fact that $S_\alpha(T(u))=S_\alpha(w)$ for all $\alpha \in I$.  But by the definitions of the $T_\alpha$ and $u$, respectively, $S_\alpha(T(u))=T_\alpha(u) = u\circ \theta_\alpha = S_\alpha(w)$ for each $\alpha \in I$.  This completes the proof.
\end{proof}

The existence and uniqueness of the maps $\theta_{\alpha,\beta}$ in our next result follows from Theorem~\ref{Thm:  Unital lattice homomorphisms} (ii). 

\begin{cor}\label{Cor:  Inverse limit is continuous functions on direct limit}
Let $I$ be an index set, let $X_\alpha$ be a realcompact space for every $\alpha \in I$, and let  $\cal{I} \defeq \left((\cont(X_\alpha))_{\alpha \in I},(T_{\beta , \alpha})_{\beta \succeq \alpha}\right)$ be an inverse system in $\vlc$ such that $T_{\beta , \alpha}$ is a unital lattice homomorphism for all $\beta \succeq \alpha$ in $I$.  For all $\alpha \preceq \beta$ in $I$, let $\theta_{\alpha , \beta}:X_\alpha\to X_\beta$ be the unique continuous map so that $T_{\beta,\alpha}(u)=T_{\theta_{\alpha,\beta}}$.  The following statements are true.
\begin{enumerate}
    \item[(i)] $\cal{I}_\star \defeq \left( (X_\alpha)_{\alpha \in I} , (\theta_{\alpha , \beta})_{\alpha \preceq \beta}\right)$ is a direct system in $\topc$.

    \item[(ii)] Take a direct limit $(X,(\theta_\alpha)_{\alpha\in I})$ of $\cal{I}_\star$ in $\topc$, and set $T_\alpha \defeq T_{\theta_\alpha}$ for every $\alpha\in I$. Then $\left( \contX,(T_\alpha)_{\alpha \in I}\right)$ is an inverse limit of $\cal{I}$ in $\vlc$.
\end{enumerate}
\end{cor}

\begin{proof}
For the proof of (i), take $\alpha \preceq \beta \preceq \gamma$ in $I$. Then an appeal to the uniqueness part of Theorem~\ref{Thm:  Unital lattice homomorphisms} (ii) readily shows that the fact that $T_{\gamma,\alpha}=T_{\beta,\alpha}\circ T_{\gamma,\beta}$ implies that $\theta_{\alpha,\gamma}=\theta_{\beta,\gamma}\circ\theta_{\alpha,\beta}$.

That (ii) is true follows immediately from (i) and Proposition \ref{Prop:  Continuous functions on direct limit is inverse limit}.
\end{proof}

\subsection{A representation theorem}\label{Subsection:  A representation theorem}

After the preparations made in the Sections~\ref{Subsection:  Direct limits of spaces of measures} and~\ref{Subsection:  Inverse limits of C(X)-spaces}, we can now establish our main result. It should be compared to the result in the introduction about the bidual of $\contK$ for a compact space $K$.

Before we proceed to state and prove our main result, we recapitulate the known results concerning the Arens product on $\obidual{\contX}$ mentioned in the introduction:  The order bidual of $\contX$ equipped with the Arens product, $(\obidual{\contX},\diamond)$ is a unital and commutative $f$-algebra.  The canonical embedding $\sigma_X:\contX\to \obidual{\contX}$ is a lattice and algebra embedding, and $\sigma_X(\onefunction_X)$ is the multiplicative identity in $(\obidual{\contX},\diamond)$.

\begin{thm}\label{Thm:  Main result}
Let $X$ be a realcompact space. The following statements are true.
\begin{enumerate}
    \item[(i)] There exist a realcompact, extremally disconnected space $\tilde{X}$, unique up to homeomorphism, such that $(\obidual{\contX},\diamond)$ and $\cont(\tilde{X})$ are isomorphic $f$-algebras.
    \item[(ii)] For every compact subset $K$ of $X$, there exists a continuous map $\eta_K:\tilde{K} \to \tilde{X}$ from its hyper-Stonean envelope $\tilde K$ into $\tilde X$ which is a homeomorphism onto its range.
    \item[(iii)] Identify $X$ with $\{\delta_x ~:~x\in X\}\subseteq \measuresc(X)$, and $\tilde{X}$ with $\{\delta_{\tilde{x}} ~:~\tilde{x}\in \tilde{X}\}\subseteq \measuresc(\tilde{X})$.  The topologies on $X$ and $\tilde{X}$, respectively, are the subspace topologies with respect to the weak$^*$-topologies $\sigma(\measuresc(X),\contX)$ and $\sigma(\measuresc(\tilde{X}),\cont(\tilde{X}))$, respectively.
    \item[(iv)] Let $\omega:\measuresc(X)\to \measuresc(\tilde{X})$ be the canonical embedding of $\measuresc(X)$ into its bidual.  Then $\omega$ maps $X$ bijectively onto the isolated points in $\tilde{X}$.
    \item[(v)] There exists a continuous surjection $\pi_X:\tilde{X}\to X$ so that the canonical embedding $\sigma_X:\contX\to \cont(\tilde{X})$ is given by $\sigma_X(u) = u\circ\pi_X$, $u\in \contX$.
\end{enumerate}
\end{thm}

\begin{proof}
As a preparation, let $\cal{K}=\{K_\alpha ~:~ \alpha \in I\}$ be the collection of compact subsets of $X$, indexed by a directed set $I$ so that $K_\alpha \subseteq K_\beta$ if and only if $\alpha \preceq \beta$, and let $\tilde{K}_\alpha$ be the hyper-Stonean envelope of $K_\alpha$. We recall the systems $\cal{D}_X = \left((\measures(K_\alpha))_{\alpha \in I},(T_{\alpha,\beta})_{\alpha \preceq \beta}\right)$ and $\cal{S}_X=\left(\measuresc(X)_{\alpha \in I},(T_{\alpha})_{\alpha\in I}\right)$ from Definition \ref{Def:  Inverse system of measures}.

We note that $T_\alpha^\sim$ maps $\measuresc(X)^\sim$ onto $\measures(K_\alpha)^\sim$. In view of Theorem~\ref{Thm:  Riesz Representation Theorem for C(X)} and the definition of the hyper-Stonean envelope of a compact space, however, we can identify $\measuresc(X)^\sim$ with $\obidual{\contX}$, and $\measures(K_\alpha)^\sim$ with $\cont(\tilde K_\alpha)$, yielding a map $T_\alpha^\sim:\obidual{\contX}\to\cont(\tilde K_\alpha)$. Likewise, denoting by $\sigma_\alpha$ the canonical embedding of $\cont(K_\alpha)$ into its coinciding norm and order bidual, we obtain a map $\sigma_\alpha:\cont(K_\alpha)\to \cont(\tilde{K}_\alpha)$.
	
\emph{We prove} (i). The first step is to find a realcompact space $\tilde X$ and a lattice isomorphism $T:\obidual{\contX}\to\cont(\tilde X)$.

According to Proposition \ref{Prop:  Direct systems of Radon measures} (ii), $\cal{D}_X$ is a direct system in $\vlic$, and $\ind{\cal{D}_X}=\left(\measuresc(X)_{\alpha \in I},(T_{\alpha})_{\alpha\in I}\right)$ in $\vlic$.  As established in Corollary \ref{Cor:  Order dual of compactly supported Radon measures}, $\cal{D}_{X}^\sim = \left( (\cont(\tilde{K}_\alpha))_{\alpha \in I},(T_{\alpha,\beta}^\sim)_{\beta \succeq \alpha}\right)$ is an inverse system in $\nvlc$, and $\proj{\cal{D}_{X}^\sim}=\left(\obidual{\contX},(T_\alpha^\sim)_{\alpha \in I}\right)$ in $\nvlc$.  We note that, according to Theorem \ref{Thm:  Existence of limits} (iii) and the essential uniqueness of inverse limits, we also have $\proj{\cal{D}_{X}^\sim}= \linebreak \left(\obidual{\contX},(T_\alpha^\sim)_{\alpha \in I}\right)$ in $\vlc$.

We will now proceed to establish another expression for $\proj{\cal{D}_{X}^\sim}$.

Let $\beta \succeq \alpha$ in $I$.  We claim that the interval preserving lattice homomorphism $T_{\alpha,\beta}^\sim:\cont(\tilde{K}_\beta)\to \cont(\tilde{K}_\alpha)$ is unital.  According to Proposition \ref{Prop:  Measure extension operator adjoint}, $T_{\alpha,\beta}^\sim(\sigma_\beta(\onefunction_{K_\beta}))=\sigma_\alpha(\onefunction_{K_\alpha})$.  But
the isometric lattice and algebra isomorphism $\obidual{\cont(K_\gamma)}\to \cont(\tilde{K}_\gamma)$ maps $\sigma_\gamma(\onefunction_{K_\gamma})$ to $\onefunction_{\tilde{K}_\gamma}$ for every $\gamma \in I$.  Therefore it follows that the lattice homomorphism $T_{\alpha,\beta}^\sim:\cont(\tilde{K}_\beta)\to \cont(\tilde{K}_\alpha)$ is unital, as claimed.

It now follows from Theorem \ref{Thm:  Unital lattice homomorphisms} (ii) that, for all $\alpha \preceq \beta$ in $I$, there exists a unique continuous map $\theta_{\alpha,\beta}:\tilde{K}_\alpha \to \tilde{K}_\beta$ so that $T_{\alpha,\beta}^\sim(u)=u\circ \theta_{\alpha,\beta}$ for all $u\in \cont(\tilde{K}_\beta)$.  It follows from Corollary \ref{Cor:  Inverse limit is continuous functions on direct limit} (i) that $\cal{D}^\sim_{X \star} \defeq \left( (\tilde{K}_\alpha)_{\alpha \in I} , (\theta_{\alpha , \beta})_{\alpha \preceq \beta}\right)$ is a direct system in $\topc$.  Take a direct limit $\left(Y,(\theta_\alpha)_{\alpha \in I}\right)$ of $\cal{D}^\sim_{X \star}$ in $\topc$.  Next, take a realcompactification $\tilde{X}\defeq\kappa Y$ of $Y$ and $\kappa:Y\to \tilde X$ as in Corollary~\ref{Cor:  Realcompactification of arbitrary space}, and set $S_\alpha \defeq T_\kappa \circ T_{\theta_\alpha}$ for every $\alpha \in I$.  Corollary \ref{Cor:  Inverse limit is continuous functions on direct limit} (ii), together with the fact that $T_\kappa$ is a lattice isomorphism, implies that $\proj{\cal{D}^\sim_{X}}=\left(\cont(\tilde{X}),(S_\alpha)_{\alpha\in I}\right)$ in $\vlc$. By the essential uniqueness of inverse limits, there exists a unique lattice isomorphism $T:\obidual{\contX}\to \cont(\tilde{X})$ so that the diagram
\begin{eqnarray}
	\begin{tikzcd}[cramped]
		\obidual{\contX} \arrow[rd, "T^\sim_\alpha"'] \arrow[rr, "T"] & & \cont(\tilde{X}) \arrow[dl, "S_\alpha"]\\
		& \cont(\tilde{K}_\alpha)
	\end{tikzcd}
\end{eqnarray}\label{EQ:  Main result diagram}
commutes for every $\alpha \in I$.  The uniqueness of $\tilde{X}$ up to homeomorphism as asserted in (i) follows from Theorem \ref{Thm:  C(X) determines X}.

We proceed to show that $T$ is even a ring isomorphism.  For every $\alpha \in I$, $T_\alpha^\sim(\sigma_X(\onefunction_X))=\onefunction_{\tilde{K}_\alpha}$ by Proposition \ref{Prop:  Measure extension operator adjoint}.  Therefore, for every $\alpha \in I$,
\[
S_\alpha(T(\sigma_X(\onefunction_X)) = T_\alpha^\sim (\sigma_X(\onefunction_X))=\onefunction_{\tilde{K}_\alpha} = S_\alpha(\onefunction_{\tilde{X}}).
\]
Since the $S_\alpha$ separate the points of $\cont(\tilde X)$, we see that $T(\sigma_X(\onefunction_X)) = \onefunction_{\tilde{X}}$.  But $\sigma_X(\onefunction_X)$ is the multiplicative identity for the Arens product in $\obidual{\contX}$, so $T$ is a ring isomorphism by \cite[Corollary 5.5]{HuijsmansdePagter1984}.

That $\tilde{X}$ is extremally disconnected follows immediately from the fact that $\obidual{\contX}$, and hence $\cont(\tilde{X})$, is Dedekind complete.

\emph{We prove} (ii).  Let $K$ be a nonempty compact subset of $X$; that is, $K=K_\alpha$ for some $\alpha \in I$.  We note that the map $S_\alpha : \cont(\tilde{X})\to \cont(\tilde{K}_\alpha)$, i.e. $S_\alpha : \cont(\tilde{X})\to \cont(\tilde{K})$, introduced in the proof of (i) above is a unital lattice homomorphism.  It follows from Theorem \ref{Thm:  Unital lattice homomorphisms} (iii) that there exists a continuous map $\eta_K:\tilde{K}\to \tilde{X}$ so that $S_\alpha(u) = u\circ \eta_K$ for every $u\in \cont(\tilde{X})$.  By Proposition \ref{Prop:  Extension and retraction of measures} (v), $T_\alpha^\sim$ is surjective.  It therefore follows from the diagram (\ref{EQ:  Main result diagram}) that $S_\alpha$ is also surjective.  Therefore $\eta_K$ is a homeomorphism onto its range by Theorem \ref{Thm:  Unital lattice homomorphisms} (iii).

\emph{We prove} (iii).  This follows from \cite[3.6]{GillmanJerison1960}.

\emph{We prove} (iv).
For every $\alpha \in I$, denote by $\omega_\alpha :\measures(K_\alpha)\to \measures(\tilde{K}_\alpha)$ the canonical embedding of $\measures(K_\alpha)$ into its order bidual.  The diagram
\[
\begin{tikzcd}[cramped]
\measures(K_\alpha) \arrow[rr, "\omega_\alpha"] \arrow[dd, "T_\alpha"'] & & \measures(\tilde{K}_\alpha) \arrow[dd, "T_\alpha^{\sim\sim}"]\\
 & & \\
\measuresc(X) \arrow[rr, "\omega"']  & & \measuresc(\tilde{X})
\end{tikzcd}
\]
commutes for every $\alpha \in I$.  We note that each of the maps $T_\alpha^{\sim\sim}$ is an injective interval preserving lattice homomorphism; this follows from Proposition \ref{Prop:  Extension and retraction of measures} (iii) and (v), and Theorem \ref{Thm:  Adjoints of interval preserving vs lattice homomorphisms} (iii) and (iv).  In fact, $T_\alpha^{\sim\sim}$ is a lattice isomorphism onto an ideal in $\measuresc(\tilde{X})$.

Let $x\in X$, and select $\alpha \in I$ so that $x\in K_\alpha$.  Then $\omega(x) = T_\alpha^{\sim\sim}(\omega_\alpha(x))$.  According to \cite[Theorem 6.5.3]{DalesDashiellLauStrass2016}, $\omega_\alpha(x)\in \tilde{K}_\alpha$.  Therefore $\omega_\alpha(x)$ is an atom in $\measures(\tilde{K}_\alpha)$; consequently, $\omega(x) = T_\alpha^{\sim\sim}(\omega_\alpha(x))$ is an atom in $\measuresc(\tilde{X})$.  Recall that $T_\alpha^\sim(\sigma_X(\onefunction_X))=\onefunction_{\tilde{K}_\alpha}$.  Hence
\begin{align*}
T_\alpha^{\sim\sim}(\omega_\alpha(x))(\onefunction_{\tilde{X}}) &= \omega_\alpha(x)(T_\alpha^\sim(\onefunction_{\tilde{X}})) = \omega_\alpha(x)(T_\alpha^\sim(\sigma_X(\onefunction_X)))\\
&= \omega_\alpha(x)(\onefunction_{\tilde{K}_\alpha}) = 1,
\end{align*}
so that $\omega(x)\in \tilde{X}$.

To see that $\omega(x)$ is an isolated point in $\tilde{X}$, we note that $\omega[\measuresc(X)]\subseteq \ordercontn{(\orderdual{\measuresc(X)})}=\ordercontn{\cont(\tilde{X})}$; that is, $\omega(x)$ is an order continuous point evaluation on $\cont(\tilde{X})$.  Therefore $\omega(x)$ is an isolated point in $\tilde{X}$.

Conversely, let $\tilde{x}$ be an isolated point in $\tilde{X}$.  Then $\tilde{x}\in \ordercontn{\cont(\tilde{X})}=\ordercontn{(\orderdual{\measuresc(X)})}$.  According to Corollary \ref{Cor:  Order dual of compactly supported Radon measures}, $\orderdual{\measuresc(X)}=\ordercontn{\measuresc(X)}$ so that $\tilde{x}\in \ordercontnbidual{\measuresc(X)}$.  Because $\measuresc(X)$ is an order dual space, it is perfect, see for instance \cite[Theorem 110.2]{Zaanen1983RSII}.  Therefore there exists $\mu \in \measuresc(X)$ so that $\omega(\mu)=\tilde{x}$.  Pick $\alpha \in I$ so that $S_\mu \subseteq K_\alpha$.    Then, by Proposition \ref{Prop:  Extension and retraction of measures} (iv), there exists $\mu_\alpha \in \measures(K_\alpha)$ so that $\mu =T_\alpha(\mu_\alpha)$.  Hence $\tilde{x} = \omega(T_\alpha(\mu_\alpha)) = T_\alpha^{\sim\sim}(\omega_\alpha(\mu_\alpha))$.  Because $T_\alpha^{\sim\sim}$ is a lattice isomorphism onto an ideal in $\measuresc(\tilde{X})$ and $T_\alpha^{\sim\sim}(\omega_\alpha(\mu_\alpha))$ is an atom in $\measuresc(\tilde{X})$, it follows that $\omega_\alpha(\mu_\alpha)$ is an atom in $\measures(\tilde{K}_\alpha)$.

We claim that $\omega_\alpha(\mu_\alpha)$ is an isolated point in $\tilde{K}_\alpha$. For this, we note that $\omega_\alpha[\measures(K_\alpha)]=\ordercontnbidual{\measures(K_\alpha)}=\ordercontn{\cont(\tilde{K}_\alpha)}$.  Since we have already established that $\omega_\alpha(\mu_\alpha)$ is an atom in $\measures(\tilde{K}_\alpha)$, we need therefore only show that $\omega_\alpha(\mu_\alpha)(\onefunction_{\tilde{K}_\alpha})=1$.  Because $\sigma_{K_\alpha}(\onefunction_{K_\alpha})=\onefunction_{\tilde{K}_\alpha}$ it follows that
\[
\omega_\alpha(\mu_\alpha)(\onefunction_{\tilde{K}_\alpha}) = \sigma_{K_\alpha}(\onefunction_{K_\alpha})(\mu_\alpha) = \int_{K_\alpha} \onefunction_{K_\alpha}\di\mu_\alpha.
\]
But
\[
\begin{array}{lll}
\displaystyle \int_{K_\alpha} \onefunction_{K_\alpha}\di\mu_\alpha & = & \displaystyle \int_X \onefunction_X \di T_\alpha(\mu_\alpha) \medskip \\

& = & \displaystyle T_\alpha(\mu_\alpha)(\onefunction_X) \medskip \\

& = & \displaystyle \sigma_X(\onefunction_X)(T_\alpha(\mu_\alpha)) \medskip \\

& = & \displaystyle T_\alpha^\sim (\sigma_X(\onefunction_X))(\mu_\alpha) \medskip \\

& = & \displaystyle T_\alpha^\sim (\onefunction_{\tilde{X}})(\mu_\alpha) \medskip \\

& = & \displaystyle \omega_\alpha(\mu_\alpha)(T_\alpha^\sim(\onefunction_{\tilde{X}})) \medskip \\

& = & \displaystyle T_\alpha^{\sim\sim}(\omega_\alpha(\mu_\alpha))(\onefunction_{\tilde{X}}) \medskip \\

& = & \displaystyle 1. \medskip \\
\end{array}
\]
Therefore $\omega_\alpha(\mu_\alpha)(\onefunction_{\tilde{K}_\alpha})=1$ which establishes the claim.

Since $\omega_\alpha$ is injective and maps $K_\alpha$ onto the isolated points of $\tilde{K}_\alpha$, it follows that $\mu_\alpha \in K_\alpha$.  Therefore $\mu = T_\alpha(\mu_\alpha) \in X$ so that $\tilde{x}=\omega(\mu)\in \omega[X]$.

\emph{We prove} (v).
Since $\sigma_X(\onefunction_X)$ is the multiplicative identity for the Arens product, it follows by (i) that $\sigma_X(\onefunction_X)=\onefunction_{\tilde{X}}$.  Therefore Theorem \ref{Thm:  Unital lattice homomorphisms} (ii) implies that there exists a continuous map $\pi_X:\tilde{X}\to X$ so that $\sigma_X(u) = u\circ \pi_X$ for all $u\in \contX$.   It remains to show that $\pi_X$ is surjective.  We do so by first showing that $\sigma_X^\sim:\measuresc(\tilde{X})\to \measuresc(X)$, the order adjoint of $\sigma_X$, is surjective.  Once this fact is established, we argue as follows.

It follows from \cite[Theorem 3.6.1]{Bogachev2007} that, for all $\mu \in\measuresc(\tilde{X})$ and $B\in \borel{X}$, $\sigma_X^\sim (\mu)(B) = \mu(\pi_X^{-1}(B))$.  Fix $x\in X$.  Then, by the surjectivity of $\sigma_X^\sim$, there exists $\mu\in \measuresc(\tilde{X})$ so that $\sigma_X^\sim(\mu)=\delta_x$.  For such a $\mu$,
\[
\mu(\pi_X^{-1}(x))=\sigma_X^\sim(\mu)(\{x\})=\delta_x(\{x\})=1
\]
so that $\pi_X^{-1}(x)\neq \emptyset$; that is, $x\in \pi_X[\tilde{X}]$.

Let us now show that $\sigma_X^\sim$ is surjective.  Observe that $\obidual{\measuresc(X)}=\measuresc(\tilde{X})$, and recall from (ii) that $\omega:\measuresc(X)\to\measuresc(\tilde{X})$ is the canonical embedding of $\measuresc(X)$ into its bidual.  Then $\omega\circ \sigma_X^\sim$ is the identity on $\measuresc(X)$ which shows that $\sigma_X^\sim$ is surjective.

\end{proof}

\bibliographystyle{amsplain}
\bibliography{Bidual_of_C(X)}

\end{document}